\documentclass[12pt]{article}
\usepackage[amsmath]{e-jc}

\usepackage{graphicx}
\usepackage{tikz}

\tikzset{global scale/.style={
		scale=#1,
		every node/.append style={scale=#1}
	}
}
\newcommand{\flo}[1]{\lfloor #1 \rfloor}
\newcommand{\cei}[1]{\lceil #1 \rceil}

\usepackage{array}
\usepackage{multirow}

\dateline{TBD}{TBD}{TBD}


\Copyright{The authors. Released under the CC BY-ND license (International 4.0).}

\title{Proof of a conjecture on\\ online Ramsey numbers of paths}

\author{Yanbo Zhang\thanks{Supported by the National Natural Science Foundation of China (Grant Nos. 11601527 \& 11971011).} \qquad  Yixin Zhang\\
\small School of Mathematical Sciences\\[-0.8ex]
\small Hebei Normal University\\[-0.8ex]
\small Shijiazhuang 050024, China\\
\small\tt ybzhang@hebtu.edu.cn, yxzhang@stu.hebtu.edu.cn}

\begin{document}

\maketitle


\begin{abstract}
  For two graphs $G_1$ and $G_2$, the online Ramsey number $\tilde{r}(G_1,G_2)$ is the smallest number of edges that Builder draws on an infinite empty graph to guarantee that there is either a red copy of $G_1$ or a blue copy of $G_2$, under the condition that Builder draws one edge in each round and Painter immediately colors it red or blue. For online Ramsey numbers of paths, Cyman, Dzido, Lapinskas, and Lo conjectured that $\tilde{r}(P_4, P_{\ell+1}) = \lceil(7\ell+2)/5\rceil$ for all $\ell \ge 3$ [Electron.\ J.\ Combin.\ 22 (2015) \#P1.15]. We verify the conjecture in this paper.
\end{abstract}

\section{Introduction}

The online Ramsey number is a notion from the intersection of graph Ramsey theory and combinatorial game theory. Given two graphs $G_1$ and $G_2$, the \emph{$(G_1, G_2)$-online Ramsey game} is played between Builder and Painter on a board with infinite vertices. In each round, Builder draws an edge between two nonadjacent vertices, and Painter colors it red or blue immediately. Builder's goal is to force either a red copy of $G_1$ or a blue copy of $G_2$ in as few rounds as possible, while Painter's goal is the opposite. The \emph{online Ramsey number} $\tilde{r}(G_1,G_2)$ is the smallest number of rounds that Builder can guarantee a win in the $(G_1, G_2)$-online Ramsey game.

We use the term `online' instead of `on-line' since the latter is somewhat outdated. Online Ramsey number is also called online size Ramsey number in literature, because it can be considered as an online version of the size Ramsey number. For two graphs $G_1$ and $G_2$, we denote by $G\to (G_1,G_2)$ if for any red-blue edge coloring of $G$, there is either a red copy of $G_1$ or a blue copy of $G_2$. The Ramsey number $r(G_1,G_2)$ and the size Ramsey number $\hat{r}(G_1,G_2)$ are the smallest numbers of vertices and edges, respectively, in a graph $G$ satisfying $G\to (G_1,G_2)$. Thus the size Ramsey number can be viewed as a game such that Builder draws all the needed edges in only one round to guarantee a win, while the online Ramsey number requires Builder to draw only one edge in each round. Evidently, $\tilde{r}(G_1,G_2)\le \hat{r}(G_1,G_2)\le {r(G_1,G_2) \choose 2}$.

In the classical case where both $G_1$ and $G_2$ are complete graphs, we write $\tilde{r}(m,n)=\tilde{r}(K_m,K_n)$ for simplicity. The study was initiated by Beck \cite{Beck1993Achievement} in 1993. Kurek and Ruci\'{n}ski \cite{Kurek2005Two} showed that $\tilde{r}(3,3)=8$. Then Pra\l{}at \cite{Pralat2008$R3417$} showed that $\tilde{r}(3,4)=17$. These two values are the only known nontrivial ones. For general $n$, the best bounds so far for the diagonal case are \[2^{(2-\sqrt{2})n+O(1)}\le \tilde{r}(n,n)\le n^{-c\frac{\log n}{\log\log n}}4^n,\] where $c$ is a positive constant. The upper bound is attributed to Conlon \cite{Conlon2010line}, and the lower one is attributed to Conlon, Fox, Grinshpun, and He \cite{Conlon2019Online}. These two papers also studied the off-diagonal case $\tilde{r}(m,n)$. The basic conjecture in this topic, attributed by Kurek and Ruci\'{n}ski \cite{Kurek2005Two} to R\"{o}dl, is to prove $\lim_{n\to+\infty}\tilde{r}(n,n)/\hat{r}(n,n)=0$. Conlon \cite{Conlon2010line} approached the conjecture by demonstrating that there is a subsequence $\{n_1,n_2,\ldots\}$ of the integers such that $\lim_{i\to+\infty}\tilde{r}(n_i,n_i)/\hat{r}(n_i,n_i)=0$.

Turning to sparse graphs, researchers explored the online Ramsey numbers involving paths, cycles, stars, and trees \cite{Adamski2021Online,Adamski2022Online,Cyman2014note,Cyman2015line,Dybizbanski2020line,Dzido2021Note,Gordinowicz2018Small,Grytczuk2008line,Latip2021Note,Pralat2008note,Pralat2012note,Song2022Online}. It is easy to find a strategy for Builder or Painter, but it is not easy to find out the optimal one. Thus most results have a gap between the upper or lower bounds, which is tough to close. Recall that $P_n$ and $C_n$ are a path and a cycle with $n$ vertices, respectively; $K_{1,n}$ and $nK_2$ are a star and a matching with $n$ edges, respectively. The following table lists the known exact values to our knowledge.

\begin{table}[hb]
	\label{table}
	\renewcommand\arraystretch{1.3}
	\centering
	\begin{tabular}{|c|c|c|c|}
		\hline
		Function & Domain & Value & Reference \\ \hline \hline
		\multirow{3}*{$\tilde{r}(P_m,P_n)$} & $m=3$, $n\ge 2$ & $\cei{5(n-1)/4}$ & \cite{Cyman2015line,Pralat2012note} \\ \cline{2-4}
		~ & $4\le m\le n\le 9$ & $\cei{3(n-1)/2}+m-4$ & \cite{Pralat2012note,Grytczuk2008line,Pralat2008note} \\ \cline{2-4}		
		~ & $m=4$, $10\le n\le 11$ & $\cei{7n/5}-1$ & \cite{Dzido2021Note} \\
		\hline
		\multirow{3}*{$\tilde{r}(C_m,P_n)$} & $m\ge 3$, $n=3$ & $\max\{m+2,\cei{5m/4}\}$ & \cite{Cyman2015line} \\ \cline{2-4}		
		~ & $m=4$, $4\le n\le 7$ & $\max\{2n-1,8\}$ & \cite{Adamski2022Online,Cyman2015line,Dybizbanski2020line,Litka2022Online} \\
		\cline{2-4}		
		~ & $m=4$, $n\ge 8$ & $2n-2$ & \cite{Adamski2022Online} \\
		\hline
		$\tilde{r}(K_{1,m},P_n)$ & $m=3$, $n\ge 2$ & $\flo{3n/2}$ & \cite{SongProof} \\
		\hline
		\multirow{3}*{$\tilde{r}(mK_2,P_n)$} & $m\ge 2$, $n=3$ & $\cei{3m/2}$ & \cite{Song2022Online} \\ \cline{2-4}
		~ & $m\ge 2$, $n=4$ & $\cei{9m/5}$ & \cite{Song2022Online} \\ \cline{2-4}		
		~ & $m=2,3$, $n\ge 5$ & $n+2m-4$ & \cite{Song2022Online} \\
		\hline
	\end{tabular}
\end{table}

In addition, Gordinowicz and Pra\l{}at \cite{Gordinowicz2018Small} determined some small values of $\tilde{r}(G_1,G_2)$, where both $G_1$ and $G_2$ are connected graphs on three or four vertices.

We restrict our attention to paths in this paper. A classical result on Ramsey numbers is that $r(P_m,P_n)=n+\flo{m/2}-1$ for all $n\ge m\ge 2$ \cite{gerencser1967ramsey}. But the size Ramsey numbers of paths are difficult to determine and have no exact expression except for a few small values. After a series of improvements, the state of the art is $(3.75+o(1))n\le \hat{r}(P_n,P_n)\le 74n$ \cite{Bal2022New,Dudek2017some}. Grytczuk, Kierstead, and Pra\l{}at \cite{Grytczuk2008line} showed the best known bounds on $\tilde{r}(P_m,P_n)$, which is $m+n-3\le \tilde{r}(P_m,P_n)\le 2m+2n-7$ for $m,n\ge 2$. An intriguing question is whether a tighter bound can be found. In 2015, Cyman, Dzido, Lapinskas, and Lo \cite{Cyman2015line} obtained that $\tilde{r}(P_3,P_{\ell+1})=\cei{5\ell/4}$. They also established the following bounds on $\tilde{r}(P_4,P_{\ell+1})$.

\begin{theorem}\cite{Cyman2015line}
	\label{thm:pathresult}
	For $\ell\ge 3$, we have $(7\ell+2)/5\le \tilde{r}(P_{4},P_{\ell+1})\le (7\ell+52)/5$.
\end{theorem}

They believed that the lower bound on $\tilde{r}(P_4,P_{\ell+1})$ should be tight and posed the following conjecture.

\begin{conjecture}\cite{Cyman2015line}
	\label{conj}
	For all $\ell \ge 3$, we have $\tilde{r}(P_4, P_{\ell+1}) = \lceil(7\ell+2)/5\rceil$.
\end{conjecture}

The small numbers of $\tilde{r}(P_4, P_{\ell})$ for $3\le \ell \le 9$ were calculated by Grytczuk, Kierstead, and Pra\l{}at \cite{Grytczuk2008line} and Pra\l{}at \cite{Pralat2012note}. The results for $6\le \ell\le 9$ already required the help of computer algorithms. Furthermore, Dzido and Zakrzewska \cite{Dzido2021Note} obtained two more values $\tilde{r}(P_4, P_{10})$ and $\tilde{r}(P_4, P_{11})$ without using computer algorithms. All these values agree with Conjecture \ref{conj}.

In this paper, we give a complete proof of Conjecture \ref{conj}.

\begin{theorem}
	\label{thm:mainresult}
	For all $\ell \ge 3$, we have $\tilde{r}(P_4, P_{\ell+1}) = \lceil(7\ell+2)/5\rceil$.
\end{theorem}

The lower bound follows from Theorem \ref{thm:pathresult}, whose proof presents a blocking strategy for Painter: she always colors an edge red unless doing so would create a red $P_4$ or a red cycle. This helpful strategy can also be applied elsewhere.

We need only show the upper bound here. For different integers $\ell$, the difficulty of its proof is indeed related to the remainder of $\ell$ modulo $5$. If $\ell\not\equiv\pm1\ (mod \ 5)$, it would be easier to attain the bound by induction based on the cases $\ell\equiv\pm1\ (mod \ 5)$. Thus the proof is divided into three cases: $\ell\equiv4\ (mod \ 5)$ (Section \ref{section2}), $\ell\equiv1\ (mod \ 5)$  (Section \ref{section3}), and the other cases  (Section \ref{section4}).

\section{$\tilde{r}(P_4,P_{5k})\le 7k-1$}\label{section2}

We prove that $\tilde{r}(P_4,P_{5k})\le 7k-1$ for every positive integer $k$ in this section, which is equivalent to $\tilde{r}(P_4,P_{\ell+1})\le \cei{(7\ell+2)/5}$ for $\ell\equiv4\ (mod \ 5)$ and $\ell\ge 3$. Since the case $k=1$ has already been proved in \cite{Pralat2012note}, we assume $k\ge 2$. Indeed, if we modify our following proof slightly, it works for the case $k=1$ as well. We omit it here to avoid redundancy. We also assume that there is no red $P_4$ in $7k-1$ rounds, since otherwise our proof is done. The following proof comprises three stages. Builder first creates some disjoint graphs, then extends each to a required graph, and finally connects them to a blue $P_{5k}$.

\subsection{Creating the structural units}
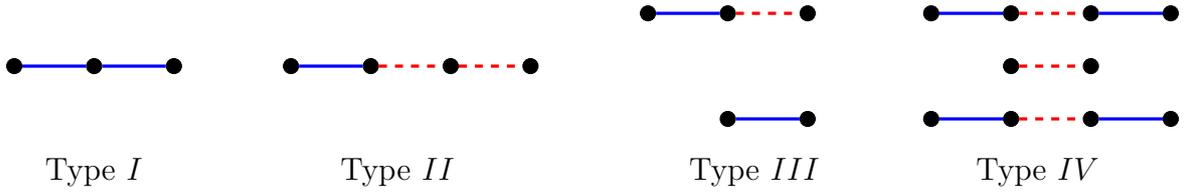
\begin{figure}[htp]
	\centering
	\begin{tikzpicture}[global scale=0.7]
		\tikzstyle{every node} = [inner sep=2pt,fill=black,circle,draw]
		\node(v1) at (0.5,1) {};
		\node(v2) at (2,1) {};
		\node(v3) at (3.5,1) {};
		\tikzstyle{every node} = [inner sep=1pt]
		\draw [blue, very thick]
		(v1)--(v2)--(v3);
		\node at (2,-1) {Type $\uppercase\expandafter{\romannumeral1}$};
	\end{tikzpicture}
	\hspace{30pt}
	\begin{tikzpicture}[global scale=0.7]
		\tikzstyle{every node} = [inner sep=2pt,fill=black,circle,draw]
		\node(w0) at (1,1) {};
		\node(w1) at (2.5,1) {};
		\node(w2) at (4,1) {};
		\node(w3) at (5.5,1) {};
		\tikzstyle{every node} = [inner sep=1pt]
		\draw [blue, very thick] 
		(w0)--(w1);
		\draw [red, dashed, very thick]
		(w1)--(w2)--(w3);
		\node at (3,-1) {Type $\uppercase\expandafter{\romannumeral2}$};
	\end{tikzpicture}
	\hspace{30pt}
	\begin{tikzpicture}[global scale=0.7]
		\tikzstyle{every node} = [inner sep=2pt,fill=black,circle,draw]
		\node(w6) at (1,2) {};
		\node(w7) at (2.5,2) {};
		\node(w8) at (4,2) {};
		\node(w9) at (2.5,0) {};
		\node(w10) at (4,0) {};
		\tikzstyle{every node} = [inner sep=1pt]
		\draw [blue, very thick] 
		(w6)--(w7)
		(w9)--(w10);
		\draw [red, dashed, very thick]
		(w7)--(w8);
		\node at (3,-1) {Type $\uppercase\expandafter{\romannumeral3}$};
	\end{tikzpicture}
	\hspace{30pt}
	\begin{tikzpicture}[global scale=0.7]
		\tikzstyle{every node} = [inner sep=2pt,fill=black,circle,draw]
		\node(w0) at (1,2) {};
		\node(w1) at (2.5,2) {};
		\node(w2) at (4,2) {};
		\node(w3) at (5.5,2) {};
		\node(w4) at (2.5,1) {};
		\node(w5) at (4,1) {};
		\node(w6) at (1,0) {};
		\node(w7) at (2.5,0) {};
		\node(w8) at (4,0) {};
		\node(w9) at (5.5,0) {};
		\tikzstyle{every node} = [inner sep=1pt]
		\draw [blue, very thick] 
		(w0)--(w1)
		(w2)--(w3)
		(w6)--(w7)
		(w8)--(w9);
		\draw [red, dashed, very thick]
		(w1)--(w2)
		(w4)--(w5)
		(w7)--(w8);
		\node at (3,-1) {Type $\uppercase\expandafter{\romannumeral4}$};
	\end{tikzpicture}
	\caption{The four types of good units.}
	\label{fig:goodunit}
\end{figure}

In the first stage, Builder aims to create a disjoint union of some special colored graphs, called structural units, which can be grouped into good units and bad units. There are four types of good units, as shown in Figure \ref{fig:goodunit}. Each of the first three types is counted as one structural unit, and the fourth type is counted as two structural units. A bad unit is either an empty graph, or one of the three colored graphs, as shown in Figure \ref{fig: badunit}. Builder will construct either $k$ good units, or $k-1$ good units and a nonempty bad unit.

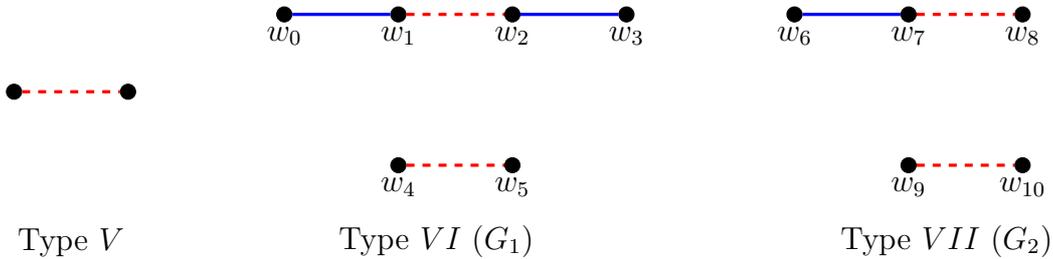
\begin{figure}[htp]
	\centering
	\begin{tikzpicture}[global scale=1]
		\tikzstyle{every node} = [inner sep=2pt,fill=black,circle,draw]
		\node(v1) at (2.25,1) {};
		\node(v2) at (3.75,1) {};
		\tikzstyle{every node} = [inner sep=1pt]
		\draw [red, dashed, very thick]
		(v1)--(v2);
		\node at (3,-1) {Type $\uppercase\expandafter{\romannumeral5}$};
	\end{tikzpicture}
	\hspace{40pt}
	\begin{tikzpicture}[global scale=1]
		\tikzstyle{every node} = [inner sep=2pt,fill=black,circle,draw]
		\node(w0) at (1,2) {};
		\node(w1) at (2.5,2) {};
		\node(w2) at (4,2) {};
		\node(w3) at (5.5,2) {};
		\node(w4) at (2.5,0) {};
		\node(w5) at (4,0) {};
		\tikzstyle{every node} = [inner sep=1pt]
		\draw [blue, very thick] 
		(w0)--(w1)
		(w2)--(w3);
		\draw [red, dashed, very thick]
		(w1)--(w2)
		(w4)--(w5);
		\draw [very thick]
		(w0)node[below=3pt]{$w_0$}
		(w1)node[below=3pt]{$w_1$} (w2)node[below=3pt]{$w_2$} (w3)node[below=3pt]{$w_3$} (w4)node[below=3pt]{$w_4$}
		(w5)node[below=3pt]{$w_5$};
		\node at (3,-1) {Type $\uppercase\expandafter{\romannumeral6}$ ($G_1$)};
	\end{tikzpicture}
	\hspace{40pt}
	\begin{tikzpicture}[global scale=1]
		\tikzstyle{every node} = [inner sep=2pt,fill=black,circle,draw]
		\node(w6) at (1,2) {};
		\node(w7) at (2.5,2) {};
		\node(w8) at (4,2) {};
		\node(w9) at (2.5,0) {};
		\node(w10) at (4,0) {};
		\tikzstyle{every node} = [inner sep=1pt]
		\draw [blue, very thick] 
		(w6)--(w7);
		\draw [red, dashed, very thick]
		(w7)--(w8)
		(w9)--(w10);
		\draw [very thick]
		(w6)node[below=3pt]{$w_6$}
		(w7)node[below=3pt]{$w_7$} (w8)node[below=3pt]{$w_8$} (w9)node[below=3pt]{$w_9$} (w10)node[below=3pt]{$w_{10}$};
		\node at (3,-1) {Type $\uppercase\expandafter{\romannumeral7}$ ($G_2$)};
	\end{tikzpicture}
	\caption{The three types of bad units.}
	\label{fig: badunit}
\end{figure}



Builder creates good units one after another. But during the process, if some special subgraph that is not a good unit shows up, we will put it in the bad unit. In the beginning, and each time after a good unit has been constructed, and each time after we place a subgraph in the bad unit, Builder draws a new $P_3$, denoted by $v_1v_2v_3$, which has three possible color patterns by symmetry: $bb$, $rr$, and $br$. If it has color pattern $bb$, then this colored path is a good unit we need. If it has color pattern $rr$, Builder draws one more edge joining $v_3$ to an isolated vertex $v_4$. The edge $v_3v_4$ has to be blue to avoid a red $P_4$. Hence the $P_4$ has color pattern $brr$ by symmetry, which is a required good unit. If $v_1v_2v_3$ has color pattern $br$ and the bad unit is empty, then Builder draws one more edge disjoint from $v_1v_2v_3$, denoted by $v_4v_5$. If $v_4v_5$ is colored blue, these three edges form a good unit we need. If $v_4v_5$ is colored red, Builder draws one more edge joining $v_3$ to an isolated vertex $v_6$. If $v_1v_2v_3v_6$ has color pattern $brr$, then it is a required good unit, and we place $v_4v_5$ in the bad unit. If $v_1v_2v_3v_6$ has color pattern $brb$, then we put both $v_1v_2v_3v_6$ and $v_4v_5$ in the bad unit. If $v_1v_2v_3$ has color pattern $br$ and the bad unit is nonempty, then Builder draws one more edge joining $v_3$ to an isolated vertex $v_6$. If $v_1v_2v_3v_6$ has color pattern $brr$, it is a required good unit. If $v_1v_2v_3v_6$ has color pattern $brb$ and the bad unit is of type $\uppercase\expandafter{\romannumeral5}$, then we put $v_1v_2v_3v_6$ in the bad unit and now the bad unit has type $\uppercase\expandafter{\romannumeral6}$. If $v_1v_2v_3v_6$ has color pattern $brb$ and the bad unit is of type $\uppercase\expandafter{\romannumeral6}$, then $v_1v_2v_3v_6$ together with the bad unit forms a good unit of type  $\uppercase\expandafter{\romannumeral4}$, which will be counted as two structural units. Then the bad unit is empty again.



When $k-2$ good units have shown up, from the above argument we see that the bad unit cannot be of type $\uppercase\expandafter{\romannumeral7}$.  Builder continues to create the next good unit. If the next one is of type $\uppercase\expandafter{\romannumeral4}$, then we have $k$ good units since the last one is counted as two units. The process now stops, and the bad unit is empty by the above analysis. If the next unit belongs to the first three types, and if the bad unit is nonempty, the process stops. If the next unit belongs to the first three types, and if the bad unit is empty, Builder continues to construct the $k$th structural unit. Builder draws two adjacent edges among isolated vertices. If both edges have the same color, then Builder can create a good unit of type $\uppercase\expandafter{\romannumeral1}$ or $\uppercase\expandafter{\romannumeral2}$ by using the same argument as above. If the two edges have different colors, Builder draws a new edge that is disjoint from the other edges. If the last edge is blue, these three edges form a good unit of type $\uppercase\expandafter{\romannumeral3}$. If the last edge is red, these three edges form a bad unit of type $\uppercase\expandafter{\romannumeral7}$. To summarize, Builder can always create either $k$ good units, or $k-1$ good units and a nonempty bad unit.

\subsection{Extending each structural unit}

In this stage, Builder extends each structural unit to a new graph such that we can connect them to force a blue $P_{5k}$ in the next stage. For bad units of type $\uppercase\expandafter{\romannumeral6}$ and type $\uppercase\expandafter{\romannumeral7}$, there is no need to extend them. In other words, the two graphs themselves are the required graphs, denoted by $G_1$ and $G_2$, respectively.


For a good unit of type $\uppercase\expandafter{\romannumeral1}$, denote this graph by $v_1v_2v_3$. Builder joins $v_1$ to an isolated vertex $v_0$ and $v_3$ to an isolated vertex $v_4$. The graph $v_0v_1v_2v_3v_4$ has three possible color patterns by symmetry: $bbbb$, $rbbr$, and $rbbb$. As the first color pattern, the blue $P_5$ is a required graph, denoted by $G_3$. For the second color pattern, Builder draws $v_0v_3$ and $v_1v_4$, both of which are blue to avoid a red $P_4$. This is a required graph, denoted by $G_4$. For the third color pattern, Builder joins $v_4$ to an isolated vertex $v_5$. If $v_4v_5$ is red, Builder then draws $v_0v_4$, which has to be blue. This is a colored graph we need, denoted by $G_5$. If $v_4v_5$ is blue, this is another colored graph we need, denoted by $G_6$.

\begin{figure}[htp]
	\centering
	\begin{tikzpicture}[global scale=1.2]
		\tikzstyle{every node} = [inner sep=2pt,fill=black,circle,draw]
		\node(v0) at (1,0) {};
		\node(v1) at (2,0) {};
		\node(v2) at (3,0) {};
		\node(v3) at (4,0) {};
		\node(v4) at (5,0) {};
		\tikzstyle{every node} = [inner sep=1pt]
		\draw [blue, very thick] 
		(v0)--(v1)--(v2)--(v3)--(v4);
		\draw [very thick]
		(v0)node[below=3pt]{$v_0$}
		(v1)node[below=3pt]{$v_1$} (v2)node[below=3pt]{$v_2$} (v3)node[below=3pt]{$v_3$} (v4)node[below=3pt]{$v_4$};
		\node at (3,-0.8) {$G_{3}$};
	\end{tikzpicture}
	\hspace{74pt}
	\begin{tikzpicture}[global scale=1.2]
		\tikzstyle{every node} = [inner sep=2pt,fill=black,circle,draw]
		\node(v0) at (1,0) {};
		\node(v1) at (2,0) {};
		\node(v2) at (3,0) {};
		\node(v3) at (4,0) {};
		\node(v4) at (5,0) {};
		\tikzstyle{every node} = [inner sep=1pt]
		\draw [blue, very thick] 
		(v1)--(v2)--(v3)
		(v0) to [bend left] (v3)
		(v1) to [bend left] (v4);
		\draw [red, dashed, very thick]
		(v0)--(v1)
		(v3)--(v4)		;
		\draw [very thick]
		(v0)node[below=3pt]{$v_0$}
		(v1)node[below=3pt]{$v_1$} (v2)node[below=3pt]{$v_2$} (v3)node[below=3pt]{$v_3$} (v4)node[below=3pt]{$v_4$};
		\node at (3,-0.8) {$G_{4}$};
	\end{tikzpicture}
	
	\vspace{1em}
	\begin{tikzpicture}[global scale=1.2]
		\tikzstyle{every node} = [inner sep=2pt,fill=black,circle,draw]
		\node(v0) at (1,0) {};
		\node(v1) at (2,0) {};
		\node(v2) at (3,0) {};
		\node(v3) at (4,0) {};
		\node(v4) at (5,0) {};
		\node(v5) at (6,0) {};
		\tikzstyle{every node} = [inner sep=1pt]
		\draw [blue, very thick] 
		(v1)--(v2)--(v3)--(v4)
		(v0) to [bend left] (v4);
		\draw [red, dashed, very thick]
		(v0)--(v1)
		(v4)--(v5);
		\draw [very thick]
		(v0)node[below=3pt]{$v_0$}
		(v1)node[below=3pt]{$v_1$} (v2)node[below=3pt]{$v_2$} (v3)node[below=3pt]{$v_3$} (v4)node[below=3pt]{$v_4$}
		(v5)node[below=3pt]{$v_5$};
		\node at (3.5,-0.8) {$G_{5}$};
	\end{tikzpicture}
	\hspace{40pt}
	\begin{tikzpicture}[global scale=1.2]
		\tikzstyle{every node} = [inner sep=2pt,fill=black,circle,draw]
		\node(v0) at (1,0) {};
		\node(v1) at (2,0) {};
		\node(v2) at (3,0) {};
		\node(v3) at (4,0) {};
		\node(v4) at (5,0) {};
		\node(v5) at (6,0) {};
		\tikzstyle{every node} = [inner sep=1pt]
		\draw [blue, very thick] 
		(v1)--(v2)--(v3)--(v4)--(v5);
		\draw [red, dashed, very thick]
		(v0)--(v1);
		\draw [very thick]
		(v0)node[below=3pt]{$v_0$}
		(v1)node[below=3pt]{$v_1$} (v2)node[below=3pt]{$v_2$} (v3)node[below=3pt]{$v_3$} (v4)node[below=3pt]{$v_4$}
		(v5)node[below=3pt]{$v_5$};
		\node at (3.5,-0.8) {$G_{6}$};
	\end{tikzpicture}
	\caption{The graphs $G_i$ for $3\le i\le 6$.}
	\label{fig:G3G6}
\end{figure}
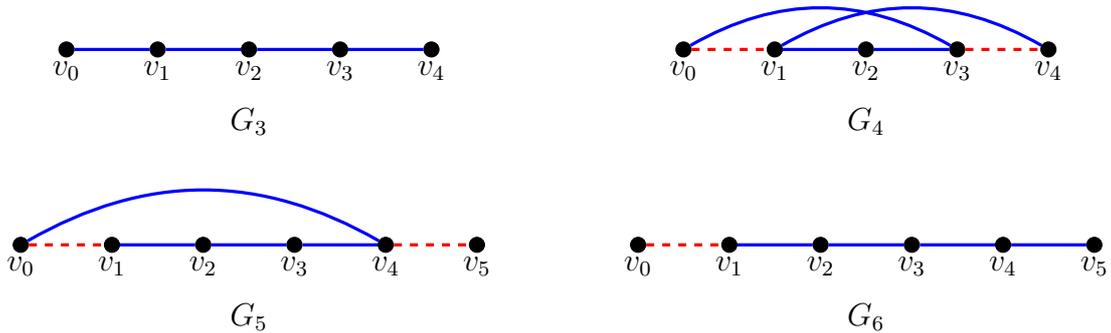

For a good unit of type $\uppercase\expandafter{\romannumeral2}$, denote this graph by $v_0v_1v_2v_3$ with $v_0v_1$ blue. Builder joins both $v_1$ and $v_3$ to an isolated vertex $v_4$. Both edges $v_1v_4$ and $v_3v_4$ have to be blue to avoid a red $P_4$. The colored graph is one of what we need, denoted by $G_7$.

For a good unit of type $\uppercase\expandafter{\romannumeral3}$, denote the two blue edges by $v_0v_1$ and $v_3v_4$, and the red edge by $v_1v_2$. Builder joins two vertices $v_2$ and $v_3$. If $v_2v_3$ is red, Builder draws $v_1v_4$, which must be blue. This is also the graph $G_7$. If $v_2v_3$ is blue, Builder joins $v_0$ to $v_4$ in the next move. If $v_0v_4$ is blue, this is a required graph, denoted by $G_8$. If $v_0v_4$ is red, Builder then joins $v_1$ to $v_4$. The edge $v_1v_4$ has to be blue and we obtain a required graph again, denoted by $G_9$.

\begin{figure}[htp]
	\centering
	\begin{tikzpicture}[global scale=1]
		\tikzstyle{every node} = [inner sep=2pt,fill=black,circle,draw]
		\node(v0) at (1,0) {};
		\node(v1) at (2,0) {};
		\node(v2) at (3,0) {};
		\node(v3) at (4,0) {};
		\node(v4) at (5,0) {};
		\tikzstyle{every node} = [inner sep=1pt]
		\draw [blue, very thick] 
		(v0)--(v1)
		(v3)--(v4)
		(v1) to [bend left] (v4);
		\draw [red, dashed, very thick]
		(v1)--(v2)--(v3);
		\draw [very thick]
		(v0)node[below=3pt]{$v_0$}
		(v1)node[below=3pt]{$v_1$} (v2)node[below=3pt]{$v_2$} (v3)node[below=3pt]{$v_3$} (v4)node[below=3pt]{$v_4$};
		\node at (3,-1) {$G_{7}$};
	\end{tikzpicture}
	\hspace{20pt}
	\begin{tikzpicture}[global scale=1]
		\tikzstyle{every node} = [inner sep=2pt,fill=black,circle,draw]
		\node(v0) at (1,0) {};
		\node(v1) at (2,0) {};
		\node(v2) at (3,0) {};
		\node(v3) at (4,0) {};
		\node(v4) at (5,0) {};
		\tikzstyle{every node} = [inner sep=1pt]
		\draw [blue, very thick] 
		(v0)--(v1)
		(v2)--(v3)--(v4)
		(v0) to [bend left] (v4);
		\draw [red, dashed, very thick]
		(v1)--(v2);
		\draw [very thick]
		(v0)node[below=3pt]{$v_0$}
		(v1)node[below=3pt]{$v_1$} (v2)node[below=3pt]{$v_2$} (v3)node[below=3pt]{$v_3$} (v4)node[below=3pt]{$v_4$};
		\node at (3,-1) {$G_{8}$};
	\end{tikzpicture}
	\hspace{20pt}
	\begin{tikzpicture}[global scale=1]
		\tikzstyle{every node} = [inner sep=2pt,fill=black,circle,draw]
		\node(v0) at (1,0) {};
		\node(v1) at (2,0) {};
		\node(v2) at (3,0) {};
		\node(v3) at (4,0) {};
		\node(v4) at (5,0) {};
		\tikzstyle{every node} = [inner sep=1pt]
		\draw [blue, very thick] 
		(v0)--(v1)
		(v2)--(v3)--(v4)
		(v1) ..controls (3,-0.7) and (4,-0.7).. (v4);
		\draw [red, dashed, very thick]
		(v1)--(v2)
		(v0) to [bend left] (v4);
		\draw [very thick]
		(v0)node[below=3pt]{$v_0$}
		(v1)node[below=3pt]{$v_1$} (v2)node[below=3pt]{$v_2$} (v3)node[below=3pt]{$v_3$} (v4)node[below=3pt]{$v_4$};
		\node at (3,-1) {$G_{9}$};
	\end{tikzpicture}
	\caption{The graphs $G_i$ for $7\le i\le 9$.}
	\label{fig:G7G9}
\end{figure}
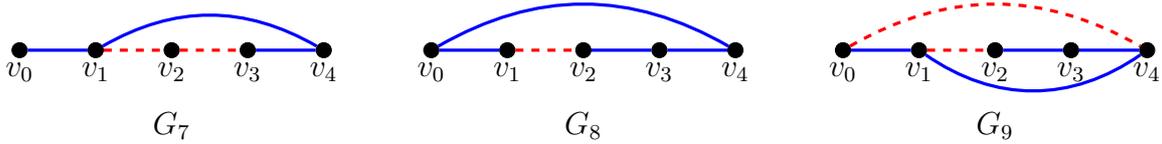

For a good unit of type $\uppercase\expandafter{\romannumeral4}$, denote the two $P_4$ by $v_0v_1v_2v_3$ and $v_4v_5v_6v_7$, and the red edge by $v_8v_9$. Builder draws two edges $v_3v_8$ and $v_7v_9$. At least one of the two edges is blue. If only one of them is blue, say, $v_7v_9$ blue and $v_3v_8$ red, then Builder draws four edges $v_0v_3$, $v_2v_8$, $v_5v_8$, and $v_4v_9$, all of which have to be blue. This is a colored graph we need, denoted by $G_{10}$. Notice that the path  $v_1v_0v_3v_2v_8v_5v_4v_9v_7v_6$ is a blue path of order ten with both ends incident to a red edge. If both edges $v_3v_8$ and $v_7v_9$ are blue, in the next three moves, Builder draws edges $v_1v_8$, $v_2v_9$, and $v_0v_4$. The first two edges have to be blue. If $v_0v_4$ is also blue, this is a colored graph we need, denoted by $G_{11}$. In this case we also obtain a blue path  $v_5v_4v_0v_1v_8v_3v_2v_9v_7v_6$ of order ten with both ends incident to a red edge. If $v_0v_4$ is red, Builder then draws $v_0v_5$, which must be blue. This is a colored graph we need, denoted by $G_{12}$. Again, we obtain a blue path $v_4v_5v_0v_1v_8v_3v_2v_9v_7v_6$ of order ten with both ends incident to a red edge.

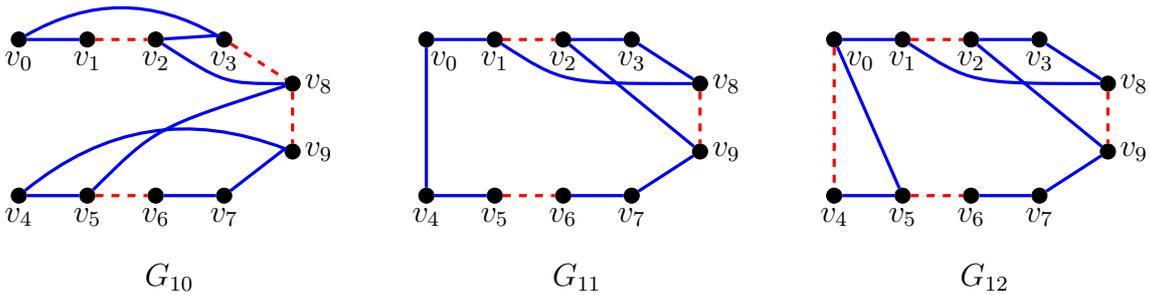
\begin{figure}[htp]
	\centering
	\begin{tikzpicture}[global scale=0.9]
		\tikzstyle{every node} = [inner sep=2pt,fill=black,circle,draw]
		\node(v0) at (1,2.3) {};
		\node(v1) at (2,2.3) {};
		\node(v2) at (3,2.3) {};
		\node(v3) at (4,2.3) {};
		\node(v4) at (1,0) {};
		\node(v5) at (2,0) {};
		\node(v6) at (3,0) {};
		\node(v7) at (4,0) {};
		\node(v8) at (5,1.65) {};
		\node(v9) at (5,0.65) {};
		\tikzstyle{every node} = [inner sep=1pt]
		\draw [blue, very thick] 
		(v1)--(v0) to [bend left] (v3)--(v2)..controls (4,1.65)..(v8)..controls (3,1)..(v5)--(v4)to [bend left](v9)--(v7)--(v6);
		\draw [red, dashed, very thick]
		(v1)--(v2)
		(v5)--(v6)
		(v3)--(v8)--(v9);
		\draw [very thick]
		(v0)node[below=3pt]{$v_0$}
		(v1)node[below=3pt]{$v_1$} (v2)node[below=3pt]{$v_2$} (v3)node[below=3pt]{$v_3$} (v4)node[below=3pt]{$v_4$}
		(v5)node[below=3pt]{$v_5$} (v6)node[below=3pt]{$v_6$} (v7)node[below=3pt]{$v_7$}
		(v8)node[right=3pt]{$v_8$} (v9)node[right=3pt]{$v_9$};
		\node at (3.2,-1.2) {$G_{10}$};
	\end{tikzpicture}
	\hspace{20pt}
	\begin{tikzpicture}[global scale=0.9]
		\tikzstyle{every node} = [inner sep=2pt,fill=black,circle,draw]
		\node(v0) at (1,2.3) {};
		\node(v1) at (2,2.3) {};
		\node(v2) at (3,2.3) {};
		\node(v3) at (4,2.3) {};
		\node(v4) at (1,0) {};
		\node(v5) at (2,0) {};
		\node(v6) at (3,0) {};
		\node(v7) at (4,0) {};
		\node(v8) at (5,1.65) {};
		\node(v9) at (5,0.65) {};
		\tikzstyle{every node} = [inner sep=1pt]
		\draw [blue, very thick] 
		(v6)--(v7)--(v9)--(v2)--(v3)--(v8)..controls (3,1.65)..(v1)--(v0)--(v4)--(v5);
		\draw [red, dashed, very thick]
		(v1)--(v2)
		(v5)--(v6)
		(v8)--(v9);
		\draw [very thick]
		(v0)node[below=8pt, right]{$v_0$}
		(v1)node[below=3pt]{$v_1$} (v2)node[below=3pt]{$v_2$} (v3)node[below=3pt]{$v_3$} (v4)node[below=3pt]{$v_4$}
		(v5)node[below=3pt]{$v_5$} (v6)node[below=3pt]{$v_6$} (v7)node[below=3pt]{$v_7$}
		(v8)node[right=3pt]{$v_8$} (v9)node[right=3pt]{$v_9$};
		\node at (3.2,-1.2) {$G_{11}$};
	\end{tikzpicture}
	\hspace{20pt}
	\begin{tikzpicture}[global scale=0.9]
		\tikzstyle{every node} = [inner sep=2pt,fill=black,circle,draw]
		\node(v0) at (1,2.3) {};
		\node(v1) at (2,2.3) {};
		\node(v2) at (3,2.3) {};
		\node(v3) at (4,2.3) {};
		\node(v4) at (1,0) {};
		\node(v5) at (2,0) {};
		\node(v6) at (3,0) {};
		\node(v7) at (4,0) {};
		\node(v8) at (5,1.65) {};
		\node(v9) at (5,0.65) {};
		\tikzstyle{every node} = [inner sep=1pt]
		\draw [blue, very thick] 
		(v6)--(v7)--(v9)--(v2)--(v3)--(v8)..controls (3,1.65)..(v1)--(v0)--(v5)--(v4);
		\draw [red, dashed, very thick]
		(v0)--(v4)
		(v1)--(v2)
		(v5)--(v6)
		(v8)--(v9);
		\draw [very thick]
		(v0)node[below=8pt, right=3pt]{$v_0$}
		(v1)node[below=3pt]{$v_1$} (v2)node[below=3pt]{$v_2$} (v3)node[below=3pt]{$v_3$} (v4)node[below=3pt]{$v_4$}
		(v5)node[below=3pt]{$v_5$} (v6)node[below=3pt]{$v_6$} (v7)node[below=3pt]{$v_7$}
		(v8)node[right=3pt]{$v_8$} (v9)node[right=3pt]{$v_9$};
		\node at (3.2,-1.2) {$G_{12}$};
	\end{tikzpicture}
	\caption{The graphs $G_i$ for $10\le i\le 12$.}
	\label{fig:G10G12}
\end{figure}

For a bad unit of type $\uppercase\expandafter{\romannumeral5}$, denote this red edge by $v_0v_1$. Builder joins $v_1$ to an isolated vertex $v_2$. If $v_1v_2$ is red, Builder then joins $v_2$ to an isolated vertex $v_3$. The edge $v_2v_3$ has to be blue to avoid a red $P_4$. Hence this is a good unit of type $\uppercase\expandafter{\romannumeral2}$. Builder extends it to $G_7$ further. If $v_1v_2$ is blue, Builder draws a new edge $v_3v_4$ disjoint from the other edges. If $v_3v_4$ is blue, this is a good unit of type $\uppercase\expandafter{\romannumeral3}$. Then Builder extends it to a required graph further. If $v_3v_4$ is red, this is a bad unit of type $\uppercase\expandafter{\romannumeral7}$, which is $G_2$.

\subsection{Connecting the structural units}

Builder has constructed twelve types of graphs in the second stage, which are denoted by $G_i$ for $1\le i\le 12$. Assume that the graph $G_i$ appears $c_i$ times. Thus we have $$\sum_{i=1}^{9}c_i+2\sum_{i=10}^{12}c_i=k.$$ We also have the restriction that $c_1+c_2\le 1$ since there is at most one bad unit. Builder connects the graphs to force a blue $P_{5k}$ in this stage.

Each of the graphs $G_4$, $G_5$, $G_8$, and $G_9$ contains a blue path $P_5$ with both ends incident to a red edge. In addition, each of the graphs $G_{10}$, $G_{11}$, and $G_{12}$ contains a blue path $P_{10}$ with both ends incident to a red edge. So we arrange these graphs which belong to the above seven types in an arbitrary order. Assume that there are $t$ such graphs. Let $u_1^i$ and $u_2^i$ denote the ends of the longest blue path ($P_5$ or $P_{10}$) in the $i$th graph for each $i$ with $1\le i\le t$. For $1\le i\le t-1$, Builder joins $u_2^i$ to $u_1^{i+1}$. Each such edge has to be blue to avoid a red $P_4$. Set $$c=(c_4+c_5+c_8+c_9)+2(c_{10}+c_{11}+c_{12}).$$ Thus we have a blue path of order $5c$, denoted by $P_{5c}$, both of whose ends $u_1^1$ and $u_2^t$ are incident to a red edge.

For the graphs which are isomorphic to $G_7$, we arrange them in an arbitrary order. For $1\le i\le c_7$, let $v_0^iv_1^iv_4^iv_3^i$ be the blue $P_4$ and $v_1^iv_2^iv_3^i$ the red $P_3$ of the $i$th graph. If $c_7\ge 2$, Builder draws $v_3^iv_0^{i+1}$ and $v_2^iv_2^{i+1}$ for each $i$ with $1\le i\le c_7-1$. To avoid a red $P_4$, these $2c_7-2$ edges must be blue. Now there are two disjoint blue paths. We use $v_0^1P_{4c_7}v_3^{c_7}$ to denote the blue path of order $4c_7$ with two ends $v_0^1$ and $v_3^{c_7}$, and $v_2^1P_{c_7}v_2^{c_7}$ to denote the blue path of order $c_7$ with two ends $v_2^1$ and $v_2^{c_7}$.

Now we see how to connect all graphs except those that are isomorphic to $G_3$ or $G_6$. We use $w_i$ for $0\le i\le 5$ to denote the vertices of $G_1$, and $w_i$ for $6\le i\le 10$ to denote the vertices of $G_2$. Their adjacency relations are shown in Figure \ref{fig: badunit}.

When $c_1+c_2+c_7=0$, there are two subcases. If $c=0$, all graphs are isomorphic to either $G_3$ or $G_6$, and the result is trivially true. If $c\ge 1$, we have obtained a blue path $P_{5c}$ as above.

When $c_1+c_2+c_7=1$, there are three subcases. If $c_7=1$, then $c_1=c_2=0$. If we further have $c\ge 1$, Builder draws two edges $v_3^1u_1^1$ and $u_2^tv_2^1$, both of which must be blue. 
Hence $v_0^1v_1^1v_4^1v_3^1u_1^1P_{5c}u_2^tv_2^1$ is a blue path of order $5c+5$. If we have $c_7=1$ and $c=0$, it follows from $k\ge 2$ that $c_3+c_6\ge 1$. We do not deal with this special case here and leave it 
to the end of this section. If $c_1=1$, then $c_2=c_7=0$. If $c\ge 1$, Builder draws three edges $w_1w_4$, $w_4u_1^1$, and $u_2^tw_2$, all of which are forced to be blue. 
Hence $w_0w_1w_4u_1^1P_{5c}u_2^tw_2w_3$ is a blue path of order $5c+5$. If $c=0$, Builder draws two edges $w_1w_4$ and $w_4w_2$, which must be blue. It follows that $w_0w_1w_4w_2w_3$ is a blue path of order five. If $c_2=1$, then $c_1=c_7=0$. If $c=0$, Builder draws three edges $w_7w_9$, $w_8w_9$, and $w_8w_{10}$, all of which have to be blue. It follows that $w_6w_7w_9w_8w_{10}$ is a blue path of order five. If $c\ge 1$, Builder draws one more edge $w_{10}u_1^1$, which is forced to be blue. 
Hence $w_6w_7w_9w_8w_{10}u_1^1P_{5c}u_2^t$ is a blue path of order $5c+5$.

When $c_1+c_2+c_7\ge 2$, there are three subcases. If $c_1=c_2=0$, then $c_7\ge 2$. If $c=0$, Builder draws the edge $v_3^{c_7}v_2^1$, which is forced to be blue. 
Consequently, we have a blue path of order $5c_7$, denoted by $P_{5c_7}$, which is $v_0^1P_{4c_7}v_3^{c_7}v_2^1P_{c_7}v_2^{c_7}$. If $c\ge 1$, Builder draws one more edge $v_2^{c_7}u_1^1$, which must be blue. 
Hence $v_0^1P_{4c_7}v_3^{c_7}v_2^1P_{c_7}v_2^{c_7}u_1^1P_{5c}u_2^t$ is a blue path of order $5c+5c_7$. If $c_1=1$, then $c_2=0$ and $c_7\ge 1$. If $c=0$, Builder draws the edges $v_3^{c_7}w_0$, $w_1w_4$, $w_4v_2^1$, and $v_2^{c_7}w_2$, all of which are forced to be blue. Consequently, we have a blue path $v_0^1P_{4c_7}v_3^{c_7}w_0w_1w_4v_2^1P_{c_7}v_2^{c_7}w_2w_3$, denoted by $P_{5c_7+5}$, which has order $5c_7+5$. If $c\ge 1$, instead of joining $v_2^{c_7}$ to $w_2$, Builder draws two edges $v_2^{c_7}u_1^1$ and $u_2^tw_2$, which are forced to be blue. Thus, the edge $v_2^{c_7}w_2$ in the blue $P_{5c_7+5}$ can be replaced by $v_2^{c_7}u_1^1P_{5c}u_2^tw_2$. In this way, we have a blue path of order $5c+5c_7+5$. If $c_2=1$, then $c_1=0$ and $c_7\ge 1$. If $c=0$, Builder draws the edges $v_3^{c_7}w_6$, $w_7w_9$, $w_9w_8$, $w_8w_{10}$, and $w_{10}v_2^1$, all of which have to be blue. Consequently, we have a blue path $v_0^1P_{4c_7}v_3^{c_7}w_6w_7w_9w_8w_{10}v_2^1P_{c_7}v_2^{c_7}$, which has order $5c_7+5$. If $c\ge 1$, Builder draws one more edge $v_2^{c_7}u_1^1$, which is forced to be blue. Thus, the blue path of order $5c_7+5$ and the blue path $P_{5c}$ are connected into a longer path, which has order $5c+5c_7+5$.

Notice that all the newly added edges in this stage so far are forced to be blue, and every blue edge is on the longest blue path. Also note that for $i\in \{1,2,4,5,7,8,9\}$, each graph $G_i$ contains at most two red edges; for $10\le i\le 12$, each graph $G_i$ contains at most four red edges. Set $c'=c+c_1+c_2+c_7$. Thus we have constructed a graph with $7c'-1$ edges which contains a blue path of order $5c'$.

To connect the blue path of order $5c'$ and the graphs isomorphic to $G_3$ or $G_6$, we need the following claim.

\begin{claim}\label{G3G6claim}
	\begin{itemize}	
		\item[(a)] If there is a blue $P_\ell$ and a graph isomorphic to $G_3$ which is disjoint from the $P_\ell$, then in the next three rounds Builder can force a blue $P_{\ell+5}$.
		
		\item[(b)] If there is a blue $P_\ell$ and a graph isomorphic to $G_6$ which is disjoint from the $P_\ell$, then in the next two rounds Builder can force a blue $P_{\ell+5}$.
	\end{itemize}
\end{claim}

\begin{proof}
	(a). Denote the blue path of $G_3$ by $x_1x_2x_3x_4x_5$, and the ends of $P_\ell$ by $y_1$ and $y_2$. Builder draws three edges $x_1y_1$, $x_1y_2$, and $x_5y_2$. At least one is blue since otherwise they form a red $P_4$. Without loss of generality, let $x_1y_2$ be blue. Then $y_1P_\ell y_2x_1x_2x_3x_4x_5$ is a blue $P_{\ell+5}$.
	
	(b). Denote the path of $G_6$ by $x_0x_1x_2x_3x_4x_5$ with $x_0x_1$ red and the other edges blue. Denote the ends of $P_\ell$ by $y_1$ and $y_2$. Builder draws two edges $x_1y_1$ and $x_5y_1$. At least one is blue since otherwise $x_0x_1y_1x_5$ is a red $P_4$. Without loss of generality, let $x_1y_1$ be blue. Hence $y_2P_\ell y_1x_1x_2x_3x_4x_5$ is a blue $P_{\ell+5}$.
\end{proof}

From the above claim we observe that once Builder connects the blue path to a graph isomorphic to $G_3$ or $G_6$, the length of the longest blue path increases by five, and we have added at most seven edges including the edges of $G_3$ or $G_6$. Therefore, Builder can finally force a blue $P_{5k}$ in $7k-1$ rounds.

Recall that we have left a small case untackled, which is $c_7=1$ and $c=c_1=c_2=0$. Then we have $c_3+c_6\ge 1$ since $k\ge 2$. With the following claim this case can be solved.

\begin{claim}
	\begin{itemize}	
		\item[(a)] If there are two graphs isomorphic to $G_3$ and $G_7$ respectively, then in the next three rounds Builder can force a blue $P_{10}$.
		
		\item[(b)] If there are two graphs isomorphic to $G_6$ and $G_7$ respectively, then in the next two rounds Builder can force a blue $P_{10}$.
	\end{itemize}
\end{claim}

\begin{proof}
	(a). Denote the blue path of $G_3$ by $x_1x_2x_3x_4x_5$. Let $v_0v_1v_4v_3$ be the blue $P_4$ and $v_1v_2v_3$ the red $P_3$ of $G_7$. Builder joins $v_2$ to $x_5$. If $v_2x_5$ is blue, Builder then draws $v_3x_1$ which has to be blue. So we obtain a blue path $P_{10}$ in two rounds, which is $v_0v_1v_4v_3x_1x_2x_3x_4x_5v_2$. If $v_2x_5$ is red, Builder joins both $v_3$ and $x_5$ to an isolated vertex $v_5$. Both edges $v_3v_5$ and $x_5v_5$ have to be blue. So we obtain a blue path $P_{10}$ in three rounds, which is $v_0v_1v_4v_3v_5x_5x_4x_3x_2x_1$.
	
	(b). Denote the path of $G_6$ by $x_0x_1x_2x_3x_4x_5$ with $x_0x_1$ red and the other edges blue. Again, let $v_0v_1v_4v_3$ be the blue $P_4$ and $v_1v_2v_3$ the red $P_3$ of $G_7$. Builder draws two edges $x_1v_2$ and $x_5v_3$, both of which must be blue. Thus $v_2x_1x_2x_3x_4x_5v_3v_4v_1v_0$ is the required blue path $P_{10}$.
\end{proof}

By the above claim, after combining the two graphs $G_3$ and $G_7$ (or $G_6$ and $G_7$), we have a blue $P_{10}$ and at most $12$ edges. This result together with Claim \ref{G3G6claim} completes this small case.

\section{$\tilde{r}(P_4,P_{5k+2})\le 7k+2$}\label{section3}

Now we prove that $\tilde{r}(P_4,P_{5k+2})\le 7k+2$ for every positive integer $k$, which is equivalent to $\tilde{r}(P_4,P_{\ell+1})\le \cei{(7\ell+2)/5}$ for $\ell\equiv1\ (mod \ 5)$ and $\ell\ge 3$. Since it was shown that $\tilde{r}(P_4,P_7)=9$ \cite{Pralat2012note}, we assume $k\ge 2$. We use almost the same argument as in Section \ref{section2}. In the beginning, Builder draws some disjoint edges one by one until either one red edge shows up or three blue edges show up. We consider two cases depending on whether there is a red edge or not.

If a red edge appears, we denote it by $z_1z_2$ and keep it isolated in the first two stages. By employing precisely the same strategy as in Section \ref{section2}, Builder will construct either $k$ good units or $k-1$ good units and a nonempty bad unit. Here, the good and bad units are the same definitions as in Section \ref{section2}. Notice that if the red edge shows up in the second round, there is already a blue edge. Then the next move corresponds to Builder's second move in Section \ref{section2}. That is, he will draw an edge adjacent to the blue edge in his next move. If the red edge shows up in the third round, there are already two blue edges, denoted by $z_3z_4$ and $z_5z_6$. Then Builder joins $z_4$ to an isolated vertex $z_7$. If $z_4z_7$ is red, then these three edges form a good unit of type $\uppercase\expandafter{\romannumeral3}$. After that Builder goes on creating the next unit. If $z_4z_7$ is blue, then $z_3z_4z_7$ forms the first good unit, which is of type $\uppercase\expandafter{\romannumeral1}$. After that Builder joins $z_6$ to an isolated vertex to create the next unit. So the first stage can always be completed. The difference from the previous section is that there is an additional red edge $z_1z_2$ apart from the $k$ structural units.

Next Builder expands each unit to a required graph and then connects them to a long blue path by the same method as in Section \ref{section2}. Besides, he also needs to add the vertices $z_1$ and $z_2$ to the blue path. If $c>0$, after creating the blue path of order $5c$, which is $u_1^1P_{5c}u_2^t$, Builder draws two more edges $z_1u_1^1$ and $z_2u_2^t$, both being forced to be blue. Thus instead of having a blue path of order $5c$, we now have a blue path of order $5c+2$ with both ends incident to a red edge. Following the same argument of Section \ref{section2}, when Builder needs to join some vertex to $u_1^1$, he joins the vertex to $z_1$ instead; when he needs to join some vertex to $u_2^t$, he joins the vertex to $z_2$ instead. Thus, we can finally obtain a blue $P_{5k+2}$ in $7k+2$ rounds. If $c=0$ and $c_7>0$, let $v_1v_2v_3$ be the red $P_3$ of a $G_7$. Then Builder joins both $z_1$ and $z_2$ to the vertex $v_2$. Both $z_1v_2$ and $z_2v_2$ have to be blue. After that, when Builder needs to force the first and the second blue edge incident to $v_2$, he instead draws the first edge incident to $z_1$ and the second edge incident to $z_2$, respectively. If $c=c_7=0$ and $c_1=1$, Builder joins both $z_1$ and $z_2$ to $w_4$. If $c=c_7=0$ and $c_2=1$, he joins both $z_1$ and $z_2$ to $w_9$. Next he uses the same strategy as in the case $c_7>0$. If $c_1=c_2=c_7=c=0$, then $c_3+c_6\ge 2$ since $k\ge 2$. If there is a $G_3$, denoted by $x_1x_2x_3x_4x_5$, then Builder draws $z_1x_1$ and $z_2x_5$. The edges cannot be both red. If both edges are blue, then we have a blue $P_7$ in seven rounds. If one of them is red, say, $z_1x_1$ red, then $z_2x_5$ is blue. Builder joins $z_2$ to an isolated vertex in the next move. This edge has to be blue and hence we have a blue $P_7$ in eight rounds. If there is no $G_3$, then a $G_6$ must exist, denoted by $x_0x_1x_2x_3x_4x_5$ with $x_0x_1$ red and the other edges blue. Then Builder draws $z_1x_1$ and $z_1x_0$. Both edges have to be blue, then we have a blue $P_7$ in eight rounds, which is $x_0z_1x_1x_2x_3x_4x_5$. It follows from Claim \ref{G3G6claim} that Builder can finally connect the blue $P_7$ and all graphs isomorphic to $G_3$ or $G_6$ to form a blue $P_{5k+2}$ in at most $7k+2$ rounds. To summarize, compared with forcing a blue $P_{5k}$ in $7k-1$ rounds, Builder can always force a blue $P_{5k+2}$ by adding two or three more rounds.



If no red edge appears in the first three rounds, we denote the three blue ones by $z_1z_2$, $z_4z_5$, and $z_6z_7$. Builder then joins both $z_2$ and $z_4$ to an isolated vertex $z_3$. If both edges are blue, then $z_1z_2z_3z_4z_5$ is a blue $P_5$. By shrinking $z_3z_4z_5$ to a single vertex $z_3$, we can view this path as a good unit of type $\uppercase\expandafter{\romannumeral1}$. Since we can force a blue $P_{5k}$ in $7k-1$ rounds, we obtain a blue $P_{5k+2}$ with two more rounds in this case.

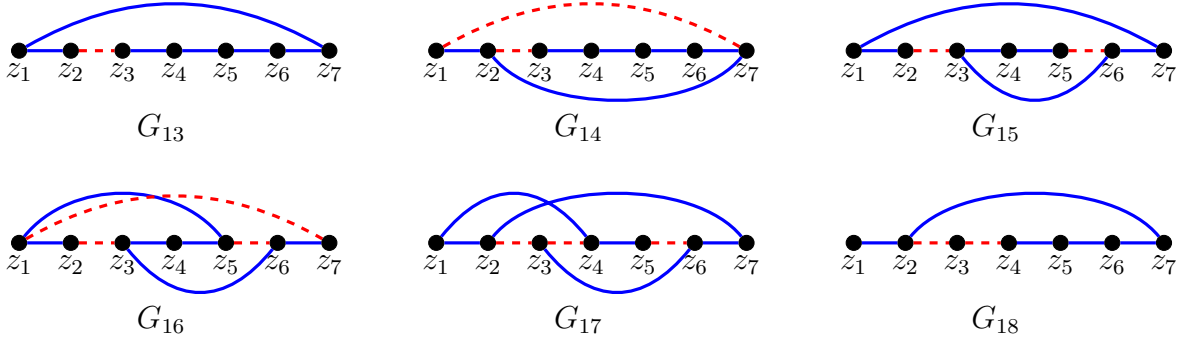
\begin{figure}[htp]
	\centering
	\begin{tikzpicture}[global scale=0.85]
		\tikzstyle{every node} = [inner sep=2pt,fill=black,circle,draw]
		\node(v1) at (1,0) {};
		\node(v2) at (1.8,0) {};
		\node(v3) at (2.6,0) {};
		\node(v4) at (3.4,0) {};
		\node(v5) at (4.2,0) {};
		\node(v6) at (5.0,0) {};
		\node(v7) at (5.8,0) {};
		\tikzstyle{every node} = [inner sep=1pt]
		\draw [blue, very thick] 
		(v1)--(v2)
		(v3)--(v4)--(v5)--(v6)--(v7)
		(v1) to [bend left] (v7);
		\draw [red, dashed, very thick]
		(v2)--(v3);
		\draw [very thick]
		(v1)node[below=3pt]{$z_1$} (v2)node[below=3pt]{$z_2$} (v3)node[below=3pt]{$z_3$} (v4)node[below=3pt]{$z_4$}
		(v5)node[below=3pt]{$z_5$} (v6)node[below=3pt]{$z_6$} (v7)node[below=3pt]{$z_7$};
		\node at (3.2,-1.2) {$G_{13}$};
	\end{tikzpicture}
	\hspace{20pt}
	\begin{tikzpicture}[global scale=0.85]
		\tikzstyle{every node} = [inner sep=2pt,fill=black,circle,draw]
		\node(v1) at (1,0) {};
		\node(v2) at (1.8,0) {};
		\node(v3) at (2.6,0) {};
		\node(v4) at (3.4,0) {};
		\node(v5) at (4.2,0) {};
		\node(v6) at (5.0,0) {};
		\node(v7) at (5.8,0) {};
		\tikzstyle{every node} = [inner sep=1pt]
		\draw [blue, very thick] 
		(v1)--(v2)
		(v3)--(v4)--(v5)--(v6)--(v7)
		(v2) ..controls (2.6,-1) and (5.0,-1).. (v7);
		\draw [red, dashed, very thick]
		(v2)--(v3)
		(v1) to [bend left] (v7);
		\draw [very thick]
		(v1)node[below=3pt]{$z_1$} (v2)node[below=3pt]{$z_2$} (v3)node[below=3pt]{$z_3$} (v4)node[below=3pt]{$z_4$}
		(v5)node[below=3pt]{$z_5$} (v6)node[below=3pt]{$z_6$} (v7)node[below=3pt]{$z_7$};
		\node at (3.2,-1.2) {$G_{14}$};
	\end{tikzpicture}
	\hspace{20pt}
	\begin{tikzpicture}[global scale=0.85]
		\tikzstyle{every node} = [inner sep=2pt,fill=black,circle,draw]
		\node(v1) at (1,0) {};
		\node(v2) at (1.8,0) {};
		\node(v3) at (2.6,0) {};
		\node(v4) at (3.4,0) {};
		\node(v5) at (4.2,0) {};
		\node(v6) at (5.0,0) {};
		\node(v7) at (5.8,0) {};
		\tikzstyle{every node} = [inner sep=1pt]
		\draw [blue, very thick] 
		(v1)--(v2)
		(v3)--(v4)--(v5)
		(v6)--(v7)
		(v1) to [bend left] (v7)
		(v3)..controls (3.4,-1) and (4.2,-1)..(v6);
		\draw [red, dashed, very thick]
		(v2)--(v3)
		(v5)--(v6);
		\draw [very thick]
		(v1)node[below=3pt]{$z_1$} (v2)node[below=3pt]{$z_2$} (v3)node[below=3pt]{$z_3$} (v4)node[below=3pt]{$z_4$}
		(v5)node[below=3pt]{$z_5$} (v6)node[below=3pt]{$z_6$} (v7)node[below=3pt]{$z_7$};
		\node at (3.2,-1.2) {$G_{15}$};
	\end{tikzpicture}
	
	\vspace{1em}
	\begin{tikzpicture}[global scale=0.85]
		\tikzstyle{every node} = [inner sep=2pt,fill=black,circle,draw]
		\node(v1) at (1,0) {};
		\node(v2) at (1.8,0) {};
		\node(v3) at (2.6,0) {};
		\node(v4) at (3.4,0) {};
		\node(v5) at (4.2,0) {};
		\node(v6) at (5.0,0) {};
		\node(v7) at (5.8,0) {};
		\tikzstyle{every node} = [inner sep=1pt]
		\draw [blue, very thick] 
		(v1)--(v2)
		(v3)--(v4)--(v5)
		(v6)--(v7)
		(v3)..controls (3.4,-1) and (4.2,-1)..(v6)
		(v1)..controls (1.8,1) and (3.4,1)..(v5);
		\draw [red, dashed, very thick]
		(v1) to [bend left] (v7)
		(v2)--(v3)
		(v5)--(v6);
		\draw [very thick]
		(v1)node[below=3pt]{$z_1$} (v2)node[below=3pt]{$z_2$} (v3)node[below=3pt]{$z_3$} (v4)node[below=3pt]{$z_4$}
		(v5)node[below=3pt]{$z_5$} (v6)node[below=3pt]{$z_6$} (v7)node[below=3pt]{$z_7$};
		\node at (3.2,-1.2) {$G_{16}$};
	\end{tikzpicture}
	\hspace{20pt}
	\begin{tikzpicture}[global scale=0.85]
		\tikzstyle{every node} = [inner sep=2pt,fill=black,circle,draw]
		\node(v1) at (1,0) {};
		\node(v2) at (1.8,0) {};
		\node(v3) at (2.6,0) {};
		\node(v4) at (3.4,0) {};
		\node(v5) at (4.2,0) {};
		\node(v6) at (5.0,0) {};
		\node(v7) at (5.8,0) {};
		\tikzstyle{every node} = [inner sep=1pt]
		\draw [blue, very thick] 
		(v1)--(v2)
		(v4)--(v5)
		(v6)--(v7)
		(v3)..controls (3.4,-1) and (4.2,-1)..(v6)
		(v2)..controls (2.6,1) and (5.0,1)..(v7)
		(v1)..controls (1.8,1) and (2.6,1)..(v4);
		\draw [red, dashed, very thick]
		(v2)--(v3)--(v4)
		(v5)--(v6);
		\draw [very thick]
		(v1)node[below=3pt]{$z_1$} (v2)node[below=3pt]{$z_2$} (v3)node[below=3pt]{$z_3$} (v4)node[below=3pt]{$z_4$}
		(v5)node[below=3pt]{$z_5$} (v6)node[below=3pt]{$z_6$} (v7)node[below=3pt]{$z_7$};
		\node at (3.2,-1.2) {$G_{17}$};
	\end{tikzpicture}
	\hspace{20pt}
	\begin{tikzpicture}[global scale=0.85]
		\tikzstyle{every node} = [inner sep=2pt,fill=black,circle,draw]
		\node(v1) at (1,0) {};
		\node(v2) at (1.8,0) {};
		\node(v3) at (2.6,0) {};
		\node(v4) at (3.4,0) {};
		\node(v5) at (4.2,0) {};
		\node(v6) at (5.0,0) {};
		\node(v7) at (5.8,0) {};
		\tikzstyle{every node} = [inner sep=1pt]
		\draw [blue, very thick] 
		(v1)--(v2)
		(v4)--(v5)--(v6)--(v7)
		(v2)..controls (2.6,1) and (5.0,1)..(v7);
		\draw [red, dashed, very thick]
		(v2)--(v3)--(v4);
		\draw [very thick]
		(v1)node[below=3pt]{$z_1$} (v2)node[below=3pt]{$z_2$} (v3)node[below=3pt]{$z_3$} (v4)node[below=3pt]{$z_4$}
		(v5)node[below=3pt]{$z_5$} (v6)node[below=3pt]{$z_6$} (v7)node[below=3pt]{$z_7$};
		\node at (3.2,-1.2) {$G_{18}$};
	\end{tikzpicture}
	\caption{The graphs being constructed in the last two subcases.}
	\label{fig:twosub}
\end{figure}

The other two subcases can be seen in Figure \ref{fig:twosub}. If two edges have different colors, without loss of generality, we may assume that $z_2z_3$ is red and $z_3z_4$ is blue. Builder then draws the edge $z_5z_6$. If $z_5z_6$ is blue, Builder joins $z_1$ to $z_7$. If $z_1z_7$ is blue, we have a blue path of order seven with both ends incident to a red edge, which is $z_3z_4z_5z_6z_7z_1z_2$. If $z_1z_7$ is red, Builder then joins $z_2$ to $z_7$, which has to be a blue edge. We now have a blue path of order seven with both ends incident to a red edge, which is $z_3z_4z_5z_6z_7z_2z_1$. If $z_5z_6$ is red, Builder draws two edges $z_3z_6$ and $z_1z_7$. The first edge has to be blue. If $z_1z_7$ is blue, we have a blue path of order seven with both ends incident to a red edge, which is $z_5z_4z_3z_6z_7z_1z_2$. If $z_1z_7$ is red, Builder then joins $z_1$ to $z_5$, which has to be a blue edge. Again, we have a blue path of order seven with both ends incident to a red edge, which is $z_2z_1z_5z_4z_3z_6z_7$. Therefore, in at most nine rounds, we can always obtain a blue $P_7$ with both ends incident to a red edge. By shrinking $z_3z_4z_5$ to a single vertex $z_3$, we may view this path as a blue $P_5$ with both ends incident to a red edge. Thus, we tackle this blue $P_7$ by the same method as we deal with the blue $P_5$ in Section \ref{section2}. It follows that we obtain a blue $P_{5k+2}$ by adding three more rounds in this case.

If both $z_2z_3$ and $z_3z_4$ are red, Builder also draws the edge $z_5z_6$ in his next move. If $z_5z_6$ is red, Builder then draws three edges $z_1z_4$, $z_2z_7$, and $z_3z_6$, all of which are forced to be blue. We have a blue path of order seven with both ends incident to a red edge, which is $z_3z_6z_7z_2z_1z_4z_5$. We use the same argument as above and then this subcase is complete. If $z_5z_6$ is blue, Builder then draws the edge $z_2z_7$, which must be blue. By shrinking $z_5z_6z_7$ to a single vertex $z_5$, we may view this graph as $G_7$. Thus, we tackle this graph by the same method as we deal with $G_7$ in Section \ref{section2}. Since we can force a blue $P_{5k}$ in $7k-1$ rounds, here we can obtain a blue $P_{5k+2}$ by adding two more rounds.

\section{The other cases}\label{section4}

We are left to prove that $\tilde{r}(P_4,P_{\ell+1})\le \cei{(7\ell+2)/5}$ for $\ell\equiv0,2,\text{or}\ 3 \ (mod \ 5)$ and $\ell\ge 3$. These cases need the following key lemma, whose proof will be given later.
\begin{lemma}\label{lem:6get4}
	During the $\tilde{r}(P_4,P_\ell)$-game, if there is a blue $P_k$, then within six rounds, Builder can force either a blue $P_{k+4}$, or a red $P_4$.
\end{lemma}

Now we deduce our main result from Lemma \ref{lem:6get4}. From the above two sections we see that $\tilde{r}(P_4, P_{5k})\le 7k-1$ and  $\tilde{r}(P_4, P_{5k+2})\le 7k+2$ for all positive integers $k$. By Lemma \ref{lem:6get4}, we have $\tilde{r}(P_4, P_{5k+4})\le \tilde{r}(P_4, P_{5k})+6\le 7k+5$, $\tilde{r}(P_4, P_{5k+6})\le \tilde{r}(P_4, P_{5k+2})+6\le 7k+8$, and $\tilde{r}(P_4, P_{5k+8})\le \tilde{r}(P_4, P_{5k+4})+6\le 7k+11$. Combining these results, we have $\tilde{r}(P_4, P_{\ell+1})\le \lceil(7\ell+2)/5\rceil$ for all $\ell \ge 8$. Note that the values for $3\le \ell\le 7$ are already known \cite{Pralat2012note}. Thus we have $\tilde{r}(P_4, P_{\ell+1})\le \lceil(7\ell+2)/5\rceil$ for all $\ell \ge 3$ and Conjecture \ref{conj} is verified.

\begin{proof}[Proof of Lemma \ref{lem:6get4}]
We assume that Painter will always avoid a red $P_4$ in the $\tilde{r}(P_4,P_\ell)$-game. Suppose that $xPy$ is a blue $P_k$ with two ends $x,y$. We need to find a blue $P_{k+4}$ in at most six rounds.

In the first two moves Builder constructs a path $P_3$ all of whose vertices are new. Then there are three possible colour patterns (up to symmetry): $bb$, $rr$, $br$. If the $P_3$ has pattern $bb$ or $rr$, then Builder extends it to a $P_4$ by joining a new vertex to any of its end vertices.
If the $P_3$ has pattern $br$, then Builder extends it to a $P_4$ by joining a new vertex to the end which is incident to its red edge. After three rounds we obtain a $P_4$. Essentially one of the four possible colour patterns appears (up to symmetry): $bbb$, $bbr$, $brb$, $brr$. Note that the pattern $rbr$ has been avoided, and the pattern $rrr$ which is a red $P_4$ cannot appear.

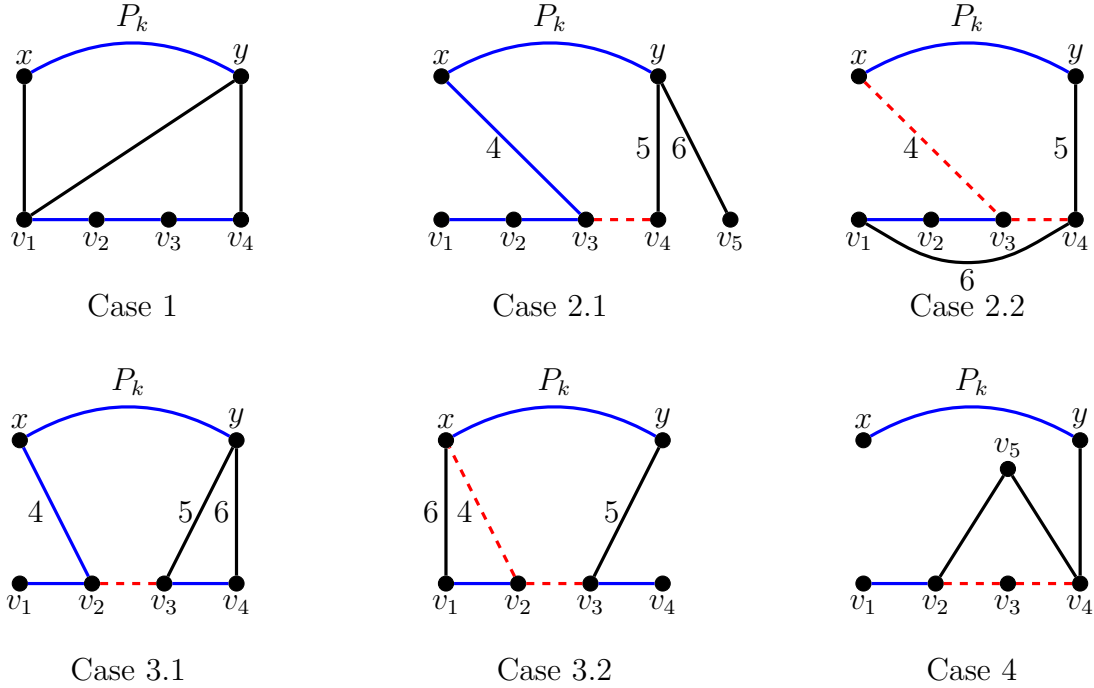
\begin{figure}[htp]
	\centering
	\begin{tikzpicture}[global scale=0.95]
		\tikzstyle{every node} = [inner sep=2pt,fill=black,circle,draw]
		\node(v1) at (1,0) {};
		\node(v2) at (2,0) {};
		\node(v3) at (3,0) {};
		\node(v4) at (4,0) {};
		\node(x) at (1,2) {};
		\node(y) at (4,2) {};
		\tikzstyle{every node} = [inner sep=1pt]
		\draw [blue, very thick] (v1)node[below=3pt, black]{$v_1$} -- (v2)node[below=3pt, black]{$v_2$} -- (v3)node[below=3pt, black]{$v_3$} -- (v4)node[below=3pt, black]{$v_4$}
		(x) to [bend left] node[above=3pt, black]{$P_k$} (y);
		\draw [very thick] (x)node[above=3pt]{$x$} -- (v1) (v4) -- (y)node[above=3pt]{$y$} (y) -- (v1);
		\begin{scope}[line width=1pt, white]
			\draw (0.8,-0.5) to (5.5,-0.5);
			\draw (0.8,-0.5) to (0.8,3.2); 
			\draw (0.8,3.2) to (5.5,3.2);
			\draw (5.5,-0.5) to (5.5,3.2); 
		\end{scope}
		\node at (2.5,-1.2) {Case 1};
	\end{tikzpicture}
	\hspace{20pt}
	\begin{tikzpicture}[global scale=0.95]
		\tikzstyle{every node} = [inner sep=2pt,fill=black, circle,draw]
		\node(v1) at (1,0) {};
		\node(v2) at (2,0) {};
		\node(v3) at (3,0) {};
		\node(v4) at (4,0) {};
		\node(v5) at (5,0) {};
		\node(x) at (1,2) {};
		\node(y) at (4,2) {};
		\tikzstyle{every node} = [inner sep=1pt]
		\draw [blue, very thick] (v1)node[below=3pt, black]{$v_1$} -- (v2)node[below=3pt, black]{$v_2$} -- (v3)node[below=3pt, black]{$v_3$}
		(x) to [bend left] node[above=3pt, black]{$P_k$} (y)
		(x)node[above=3pt, black]{$x$} --node[left=3pt, black]{$4$} (v3);
		\draw [red, dashed, very thick](v3)-- (v4)node[below=3pt, black]{$v_4$};
		\draw [very thick] (v4) --node[left=1pt]{$5$} (y)node[above=3pt]{$y$} (y) --node[left=1pt]{$6$}  (v5)node[below=3pt]{$v_5$};
		\begin{scope}[line width=1pt, white]
			\draw (0.8,-0.5) to (5.5,-0.5);
			\draw (0.8,-0.5) to (0.8,3.2); 
			\draw (0.8,3.2) to (5.5,3.2);
			\draw (5.5,-0.5) to (5.5,3.2); 
		\end{scope}
		\node at (2.5,-1.2) {Case 2.1};
	\end{tikzpicture}
	\hspace{20pt}
	\begin{tikzpicture}[global scale=0.95]
		\tikzstyle{every node} = [inner sep=2pt,fill=black,circle,draw]
		\node(v1) at (1,0) {};
		\node(v2) at (2,0) {};
		\node(v3) at (3,0) {};
		\node(v4) at (4,0) {};
		\node(x) at (1,2) {};
		\node(y) at (4,2) {};
		\tikzstyle{every node} = [inner sep=1pt]
		\draw [blue, very thick] (v1)node[below=3pt, black]{$v_1$} -- (v2)node[below=3pt, black]{$v_2$} -- (v3)node[below=3pt, black]{$v_3$}
		(x) to [bend left] node[above=3pt, black]{$P_k$} (y);
		\draw [red, dashed, very thick] (x)node[above=3pt, black]{$x$} --node[left=3pt, black]{$4$} (v3)
		(v3) -- (v4)node[below=3pt, black]{$v_4$};
		\draw [very thick] (v4) --node[left=1pt]{$5$} (y)node[above=3pt]{$y$} (v1) to [out=330, in=180] (2.5,-0.6) node[below=1pt]{$6$} to [out=0, in=210] (v4);
		\begin{scope}[line width=1pt, white]
			\draw (0.8,-0.5) to (0.8,3.2); 
			\draw (0.8,3.2) to (5.5,3.2);
			\draw (5.5,-0.5) to (5.5,3.2); 
		\end{scope}
		\node at (2.5,-1.2) {Case 2.2};
	\end{tikzpicture}

	\vspace{1em}
	\begin{tikzpicture}[global scale=0.95]
		\tikzstyle{every node} = [inner sep=2pt,fill=black,circle,draw]
		\node(v1) at (1,0) {};
		\node(v2) at (2,0) {};
		\node(v3) at (3,0) {};
		\node(v4) at (4,0) {};
		\node(x) at (1,2) {};
		\node(y) at (4,2) {};
		\tikzstyle{every node} = [inner sep=1pt]
		\draw [blue, very thick] (v1)node[below=3pt, black]{$v_1$} -- (v2)node[below=3pt, black]{$v_2$} --node[left=3pt, black]{$4$}(x) (v3)node[below=3pt, black]{$v_3$} -- 
		(v4)node[below=3pt, black]{$v_4$}
		(x)node[above=3pt, black]{$x$} to [bend left] node[above=3pt, black]{$P_k$} (y);
		\draw [red, dashed, very thick] (v2) -- (v3);
		\draw [very thick] (v4) --node[left=1pt]{$6$} (y)node[above=3pt]{$y$}
		(y) -- node[left=1pt]{$5$} (v3);
		\begin{scope}[line width=1pt, white]
			\draw (0.8,-0.5) to (5.5,-0.5);
			\draw (0.8,-0.5) to (0.8,3.2); 
			\draw (0.8,3.2) to (5.5,3.2);
			\draw (5.5,-0.5) to (5.5,3.2); 
		\end{scope}
		\node at (2.5,-1.2) {Case 3.1};
	\end{tikzpicture}
	\hspace{20pt}
	\begin{tikzpicture}[global scale=0.95]
		\tikzstyle{every node} = [inner sep=2pt,fill=black,circle,draw]
		\node(v1) at (1,0) {};
		\node(v2) at (2,0) {};
		\node(v3) at (3,0) {};
		\node(v4) at (4,0) {};
		\node(x) at (1,2) {};
		\node(y) at (4,2) {};
		\tikzstyle{every node} = [inner sep=1pt]
		\draw [blue, very thick] (v1)node[below=3pt, black]{$v_1$} -- (v2) (v3)node[below=3pt, black]{$v_3$} -- 
		(v4)node[below=3pt, black]{$v_4$}
		(x)node[above=3pt, black]{$x$} to [bend left] node[above=3pt, black]{$P_k$} (y);
		\draw [red, dashed, very thick] (v2) -- (v3)
		(v2)node[below=3pt, black]{$v_2$} --node[left=2pt, black]{$4$}(x);
		\draw [very thick] (v1) --node[left=1pt]{$6$} (x)
		(y)node[above=3pt]{$y$} -- node[left=1pt]{$5$} (v3);
		\begin{scope}[line width=1pt, white]
			\draw (0.8,-0.5) to (5.5,-0.5);
			\draw (0.8,3.2) to (5.5,3.2);
			\draw (5.5,-0.5) to (5.5,3.2); 
		\end{scope}
		\node at (2.5,-1.2) {Case 3.2};
	\end{tikzpicture}
	\hspace{20pt}
	\begin{tikzpicture}[global scale=0.95]
		\tikzstyle{every node} = [inner sep=2pt,fill=black,circle,draw]
		\node(v1) at (1,0) {};
		\node(v2) at (2,0) {};
		\node(v3) at (3,0) {};
		\node(v4) at (4,0) {};
		\node(v5) at (3,1.6) {};
		\node(x) at (1,2) {};
		\node(y) at (4,2) {};
		\tikzstyle{every node} = [inner sep=1pt]
		\draw [blue, very thick] (v1)node[below=3pt, black]{$v_1$} -- (v2) 
		(x)node[above=3pt, black]{$x$} to [bend left] node[above=3pt, black]{$P_k$} (y);
		\draw [red, dashed, very thick] (v2) -- 
		(v3)node[below=3pt, black]{$v_3$} -- 
		(v4)node[below=3pt, black]{$v_4$};
		\draw [very thick] (v4) -- (y)node[above=3pt]{$y$}
		(v2)node[below=3pt]{$v_2$} -- (v5)node[above=3pt]{$v_5$} -- (v4);
		\begin{scope}[line width=1pt, white]
			\draw (0.8,-0.5) to (5.5,-0.5);
			\draw (0.8,-0.5) to (0.8,3.2); 
			\draw (0.8,3.2) to (5.5,3.2);
			\draw (5.5,-0.5) to (5.5,3.2); 
		\end{scope}
		\node at (2.5,-1.2) {Case 4};
	\end{tikzpicture}
	\caption{Forcing a blue $P_{k+4}$ or a red $P_4$ in Lemma~\ref{lem:6get4}.}
	\label{fig:useA-path1}
\end{figure}

Set $v_1v_2v_3v_4$ be this path. If it has pattern $bbb$, in the next three moves Builder chooses the edges $xv_1$, $yv_1$, $yv_4$. At least one of them is blue, since otherwise there is a red $P_4$. Then, at least one of the paths $yPxv_1v_2v_3v_4$, $xPyv_1v_2v_3v_4$, $xPyv_4v_3v_2v_1$ is a blue $P_{k+4}$.

If $v_1v_2v_3v_4$ has pattern $bbr$, Builder chooses the edge $xv_3$. If $xv_3$ is blue, then in the next two moves Builder chooses the edges $yv_4$ and $yv_5$, where $v_5$ is a new vertex. The two edges cannot be both red, since otherwise $v_3v_4yv_5$ is a red $P_4$. Hence, either $v_1v_2v_3xPyv_4$ or $v_1v_2v_3xPyv_5$ is a blue $P_{k+4}$. If $xv_3$ is red, then in the next two moves Builder chooses the edges $yv_4$ and $v_4v_1$. To avoid a red $P_4$, both edges should be blue. Then $xPyv_4v_1v_2v_3$ is a blue $P_{k+4}$.

If $v_1v_2v_3v_4$ has pattern $brb$, Builder chooses the edge $xv_2$. If $xv_2$ is blue, then in the next two moves Builder chooses the edges $yv_3$ and $yv_4$. The two edges cannot be both red, since otherwise $v_2v_3yv_4$ is a red $P_4$. Hence, either $v_1v_2xPyv_3v_4$ or $v_1v_2xPyv_4v_3$ is a blue $P_{k+4}$. If $xv_2$ is red, then in the next two moves Builder chooses the edges $yv_3$ and $xv_1$. To avoid a red $P_4$, both edges should be blue. Then $v_2v_1xPyv_3v_4$ is a blue $P_{k+4}$.

If $v_1v_2v_3v_4$ has pattern $brr$, in the next three moves Builder chooses the edges $yv_4$, $v_2v_5$, $v_4v_5$, where $v_5$ is a new vertex. To aviod a red $P_4$, all three edges should be blue. Thus, $v_1v_2v_5v_4yPx$ is a blue $P_{k+4}$.

In each pattern, we force a blue $P_{k+4}$ in the second three steps. That is, within six rounds, Builder can force either a blue $P_{k+4}$, or a red $P_4$.
\end{proof}

\subsection*{Acknowledgements}

The authors would like to thank Ruyu Song and Sha Wang for the helpful discussions.


\end{document}